\documentclass[a4paper]{article}
\usepackage[american]{babel}
\usepackage[utf8]{inputenc}
\usepackage{cite}
\usepackage{geometry}
\usepackage{amsmath, amsthm, latexsym, amssymb}
\usepackage{mathabx}
\usepackage{amsfonts}
\usepackage{booktabs}  
\usepackage{graphicx} 
\usepackage{listings}
\usepackage{hyperref}
\usepackage{anysize}
\usepackage{caption}
\usepackage{subcaption}
\usepackage{placeins}
\usepackage{mathtools}
\usepackage{todonotes}
\usepackage{tcolorbox}
\newcommand{\ignore}[1]{}
\usepackage{comment}

\usepackage[ruled,vlined]{algorithm2e}
\usepackage{algorithmic}

\def\R{\mathbb{R}}

\def\Q{\mathbb{Q}}

\def\normal{\mathbf{n}}

\def\bp{\mathbf{p}}
\def\bq{\mathbf{q}}
\newtheorem{theorem}{Theorem}

\newtheorem{defi}{Definition}

\newtheorem{prop}{Proposition}
\newtheorem{remark}{Remark}

\newcommand{\partition}{\mathcal{T}_h}

\usepackage[wby]{callouts}

\title{A structure-preserving local discontinuous Galerkin
method for the Fokker-Planck-Landau equation}
\author{Kun Huang\thanks{Institute for Fusion Studies, University of Texas at Austin, Austin, TX 78712, USA, (k\_huang@utexas.edu).}\and Andr\'es Galindo-Olarte\thanks{Oden Institute for Computational Sciences and Engineering, University of Texas at Austin, Austin, TX 78712, USA, (afgalindo@utexas.edu).}\and Rodrigo Gonz\'alez-Hern\'andez\thanks{Oden Institute for Computational Sciences and Engineering, University of Texas at Austin, Austin, TX 78712, USA, (rodrigogonzalez@utexas.edu).} \and Irene M. Gamba\thanks{Department of Mathematics and Oden Institute for Computational Sciences and Engineering, University of Texas at Austin, Austin, TX 78712, USA, (gamba@math.utexas.edu).}}
\date{\today}

\begin{document}
\maketitle

\begin{abstract}
In this work, we introduce a structure-preserving local discontinuous Galerkin (LDG) method \cite{cockburn1998local} for solving the non-local non-linear  Fokker-Planck-Landau (FPL) equations. We rephrase the structure-preserving strategy of Shiroto and Sentoku\cite{shiroto2019structure} in the language of numerical analysis, and extend it to the LDG framework. We propose a method that is not only conservative, but also stabilized through upwind flux. The apparent contradiction between conservation laws and numerical stabilization is elegantly resolved by leveraging the properties of the jump terms inherent to the LDG framework. In the numerical experiments, our scheme is tested with benchmark examples.
\end{abstract}

\section{Introduction}\label{sec:introduction}
In this work, we consider the Fokker–Planck–Landau (FPL) equation. Developed by Lev Landau in 1936, this model introduced a correction to the Boltzmann equation designed to describe dilute hot plasmas, where charged particles interact via long-range Coulomb forces \cite{hinton1983collisional, lifshitz1981course}. Several variants of the FPL collisional operator have since emerged, by taking the grazing limit of Boltzmann equations for different interaction potentials. Moreover, when particle velocities approach the speed of light, relativistic effects become significant, and a relativistic version of the FPL equation was derived by Budker and Beliaev in 1956 \cite{beliaev1956relativistic, belyaev1961boltzmann, 1956SPhD....1..218B}. Simplified models such as the Dougherty operator\cite{dougherty1964model} and the linear Landau operator have also been extensively used in numerical simulations. In general, the FPL collision operator introduces dissipation that drives the system toward thermodynamic equilibrium.

The FPL equation has applications across multiple areas in science and engineering, including controlled fusion, solar wind and radiation belt. In particular, the study of runaway electrons in tokamaks necessitates numerical simulation based on the relativistic version\cite{dreicer1959electron}.

There are several challenges\cite{dimarco2014numerical} associated with numerically solving the Fokker-Planck-Landau (FPL) equation: high dimensionality of the phase space, non-locality and non-linearity of the collision operator, the need to preserve physical invariants (mass, momentum, and energy), positivity, numerical stability and the decay of entropy. It is generally impractical to design a single numerical method that resolves all of the aforementioned difficulties; practical schemes typically involve compromises tailored to the problem at hand.

A wide variety of deterministic numerical methods have been developed for the FPL equation. Spectral methods\cite{zhang2017conservative, pennie2019decay, pennie2019entropy, pennie2020convergence, dimarco2015numerical} benefit from low time complexity due to their use of the Fast Fourier Transform and the special convolution structure of the non-relativistic collision operator. Particle method by Carrillo et.al.\cite{carrillo2020particle} retains conservation and entropy decay by adopting a regularized collision operator. Hirvijoki et al. \cite{hirvijoki2017conservative} proposed a conservative FEM/DG scheme for by using a basis capable of exactly representing second-order polynomials. Daniel et al. \cite{daniel2020fully} proposed an implicit conservative finite difference scheme based on the Fokker–Planck–Rosenbluth formulation. Degond et al. \cite{degond1994entropy} used the logarithmic form of the equation to construct an entropy scheme. Other techniques include sublattice multigrid and Monte Carlo methods \cite{buet1997fast}, as well as finite volume schemes \cite{taitano2015mass}. For the relativistic FPL equation, Shiroto et al. \cite{shiroto2019structure} proposed a structure-preserving finite difference discretization.

In this paper we propose the use of local discontinuous Galerkin (LDG) method. Introduced in \cite{cockburn1998local} and further developed in \cite{cockburn1999some, castillo2000priori, cockburn2001superconvergence, arnold2002unified, cockburn2001runge}, it is well-suited for solving kinetic equations, as it readily supports hp-adaptivity, allows for easy implementation of stabilization techniques, and is naturally amenable to parallelization. Despite preceding works on conservative DG solvers\cite{shiroto2022mass, kim2022nonlinear}, our solver is unique in its introduction of \emph{structure-preserving} techniques. 

Structure-preserving schemes are conservative, but to be precise, the two concepts are not equivalent. Conservation can be enforced through various means—for instance, by introducing a correction step\cite{zhang2017conservative,daniel2020fully} that projects the solution back onto the conservative subspace, or by enlarging the test space\cite{shiroto2022mass,hirvijoki2017conservative} to explicitly include momentum and energy. However, such approaches do not necessarily preserve the underlying structure of the collision operator. Specifically, this structure refers to the fact that relative velocity always lies in the null space of integral kernel, which will be addressed in detail later.


The paper is organized as follows. Section \ref{sec:background} introduces FPL equations and their properties. In Section \ref{sec:num_methods}, we introduce the LDG method and the discrete gradient, along with relevant notations. The structure-preserving strategy is discussed in Section \ref{sec:conservation_strategy}. Section \ref{sec:implement} presents details associated to implementation. In Section \ref{sec:numerics}, we show numerical results that validate our conservation properties. Finally, Section \ref{sec:conclude} provides concluding remarks and outlines future directions.

\section{The Fokker-Planck-Landau Equations} \label{sec:background}

The Fokker-Planck-Landau equation has the following structure (we keep the notation consistent with Strain and Taskovic\cite{strain2019entropy})
\begin{equation}
		\partial_{t}f=\nabla_{p}\cdot\left(\int_{\R^3}\Phi(\bp,\bq)\cdot\left(f(\bq)\nabla_{p}f(\bp)-f(\bp)\nabla_{q}f(\bq)\right)\,d\bq \right).\label{eqn:fpl_mal}
\end{equation}
where $\Phi$ is a collision kernel determined by the interaction potential. If we let the single particle kinetic energy be $\mathcal{E}(\bp)$,

\begin{equation}
    \mathcal{E}(\bp) \coloneqq \left\{
    \begin{aligned}
        &\mathbf{p}^{2}/2,&\text{non-relativistic},\\
        &\sqrt{1+\mathbf{p}^{2}},&\text{relativistic}.
    \end{aligned}
    \right.
\end{equation}
this collision kernel can be written as the product of a scalar field $\Lambda$ and a tensor field $\mathbb{S}$:
\begin{equation*}
    \Phi(\bp,\bq)=\Lambda\left(\bp,\bq\right)\mathbb{S}\left(\nabla_{p}\mathcal{E}(\bp),\nabla_{q}\mathcal{E}(\bq)\right).
\end{equation*}

The tensor field $\mathbb{S}(\mathbf{v},\mathbf{w})$ depends on velocities $\mathbf{v},\mathbf{w} \in \mathbb{R}^3$. Denoting the relative velocity as $\mathbf{u}\coloneqq\mathbf{v}-\mathbf{w}$, and the mean velocity as $\mathbf{z}\coloneqq\frac{\mathbf{v}+\mathbf{w}}{2}$, then we consider tensors $\mathbb{S}$ of the following form,
\begin{equation}\label{eq:tensorS}
    \mathbb{S}(\mathbf{v},\mathbf{w}) \coloneqq \left\{
    \begin{aligned}
        &\mathbf{u}\cdot\mathbf{u}-\mathbf{u}\otimes\mathbf{u},&\text{non-relativistic},\\
        &\mathbf{u}\cdot\mathbf{u}-\mathbf{u}\otimes\mathbf{u}-\left(\mathbf{z}\times\mathbf{u}\right)\otimes\left(\mathbf{z}\times\mathbf{u}\right),&\text{relativistic}.
    \end{aligned}
    \right.
\end{equation}
This tensor is the key to energy conservation and entropy decay. Indeed, by testing against a suitable test function $\varphi(\mathbf{p})$ we obtain the weak form,
\begin{equation}\label{eq:weaklandau}
		\int_{\R^3}\varphi\partial_{t}f \,d \bp =-\int_{\R^3}\int_{\R^3}\Phi(\bp,\bq):\left[\nabla_{p}\varphi(\bp)\otimes\left(f(\bq)\nabla_{p}f(\bp)-f(\bp)\nabla_{q}f(\bq)\right)\right] \,d\bq \,d\bp.
\end{equation}
Since the collision kernel is symmetric, i.e. $\Phi(\bp,\bq)=\Phi(\bq,\bp)$, it can be rewritten as:
\begin{equation}\label{eq:weaklandau_structure}
    \int_{\mathbb{R}^3}\varphi\partial_{t}f \,d \bp =-\int_{\mathbb{R}^3}\int_{\mathbb{R}^3}\Phi(\bp,\bq):\left[\left(\nabla_{q}\varphi(\bq)-\nabla_{p}\varphi(\bp)\right)\otimes f(\bp)\nabla_{q}f(\bq)\right] \,d\bq \,d\bp .
\end{equation}

Based on Equation \eqref{eq:weaklandau_structure}, we are able to prove the following conservation laws. 
\begin{theorem}[Conservation Laws]\label{thm:conserv}
    Suppose that $f(\mathbf{p},t)$ is the solution to the FPL equation \eqref{eqn:fpl_mal}, then we have
    \begin{equation}
        \begin{split}
            \partial_{t}\left(f,1\right)&=0\\
            \partial_{t}\left(f,\mathbf{p}\right)&=0\\
            \partial_{t}\left(f,\mathcal{E}(\mathbf{p})\right)&=0
        \end{split}
    \end{equation}
\end{theorem}

\begin{proof}
    One can immediately observe that the mass $\int_{\mathbb{R}^{3}}f\ d\mathbf{p}$ and momentum $\int_{\mathbb{R}^{3}}f \mathbf{p} \ d \mathbf{p}$ are conserved since $\nabla_{q}\varphi(\bq)-\nabla_{p}\varphi(\bp)=0$ when $\varphi(\mathbf{p})=1$ or $\mathbf{p}$.

As for energy conservation, note that the vectors $\mathbf{u}$ and $\mathbf{z}$ generates a set of orthogonal basis vectors: $\mathbf{u}$, $\mathbf{z} \times \mathbf{u}$ and $\mathbf{u} \times (\mathbf{z} \times \mathbf{u})$. They are, in fact, three eigenvectors of the tensor $\mathbb{S}$ in \eqref{eq:tensorS}. In particular, the relative velocity $\mathbf{u}$ is an eigenvector whose eigenvalue is zero, i.e.
\begin{equation*}
    \mathbb{S}\cdot\mathbf{u} = \mathbb{S}\left(\nabla_{p}\mathcal{E}(\bp),\nabla_{q}\mathcal{E}(\bq)\right)\cdot\left(\nabla_{p}\mathcal{E}(\bp)-\nabla_{q}\mathcal{E}(\bq)\right)=0.
\end{equation*}
As a consequence, the FPL equation always preserves energy $\int_{\mathbb{R}^{3}}f \mathcal{E} \ d \mathbf{p}$.
\end{proof}

\begin{remark}
    In classical mechanics, the velocity of a particle is proportional to its momentum, while in the theory of relativity it is not. To be as general as possible, we use $\nabla_{p}\mathcal{E}(\bp)$ to represent velocity. And as we will see later, this relation between velocity and kinetic energy is the key to a structure-preserving strategy.
\end{remark}

Moreover, the positive semi-definiteness of tensor $\mathbb{S}$ leads to entropy dissipation
\begin{equation*}
    \partial_{t}\left(\int_{\mathbb{R}^{3}}f \log f \ d \mathbf{p}\right) \leq 0.
\end{equation*}

A key step in designing out method is distributing the tensor-vector product in the strong form (\ref{eqn:fpl_mal}) in order to obtain a \textit{non-linear, non-local, advection-diffusion equation}

\begin{equation}
    \partial_t f =\nabla_{p}\cdot \left(\mathcal{D}\nabla_{p} f-\mathcal{U}f \right),
    \label{eqn:adv_diff_form}
\end{equation}
where 
\begin{equation}\label{eqn:advection_coeff}
    \mathcal{U}(\bp)=\int_{\R^3} \Phi(\bp,\bq)\nabla_{q}f(\bq)\,d\bq,
\end{equation}
and
\begin{equation}
    \mathcal{D}(\bp)=\int_{\R^3} \Phi(\bp,\bq)f(\bq)\,d\bq. \label{eqn:diffusion_coeff}
\end{equation}
Note that the positive semi-definiteness of tensor field $\mathbb{S}$ and the non-negativity of \textit{pdf} $f$ guarantees the positive semi-definiteness of diffusion coefficient $\mathcal{D}(\mathbf{p})$.

Given certain \textit{pdf} $f(\mathbf{p})$, there can be subsets of the momentum space $\mathbb{R}^{3}$, where advection is dominant, i.e. $\lVert\mathcal{U} \rVert\gg \lVert \mathcal{D}\rVert$. Moreover, when coupled with the Vlasov equation, the magnitude of the Lorentz force can easily exceed that of the collision term. From this perspective, stabilization techniques such as upwinding becomes necessary, which motivates the use of the local discontinous Galerkin (LDG) method.

\section{The Local Discontinuous Galerkin Discretization}\label{sec:num_methods}
Here we will define the LDG method for an advection-diffusion equation in semi-discrete form. Note that the diffusion coefficient $\mathcal{D}$ is a tensor rather than a constant, therefore we take the $L^2$-projection of the diffusion flux, as was proposed in \cite{cockburn1999some}. The discrete gradient operator will be introduced, which is the key ingredient to our structure-preserving scheme.

\subsection{Notations, definitions and projections}\label{ssec:notation}

As have been proved in \cite{guillen2023landau}, the solution of a homogeneous FPL equation is always bounded by a Gaussian distribution determined by the initial condition only. Therefore given any $0<\epsilon \ll 1$, there exists a finite domain $\Omega \subsetneq\mathbb{R}^{3}$ such that  for any $t\geq0$,
\begin{equation*}
    \left|1-\frac{\int_{\Omega}f(\bp,t)d\bp}{\int_{\mathbb{R}^{3}}f(\bp,t)d\bp}\right|\leq\epsilon.
\end{equation*}
Hence it is reasonable to assume a compact support for the solution and truncate the computational domain, analogous to existing work on kinetic equations, for example \cite{zhang2017conservative, zhang2018conservative}. 

To solve on a cut-off domain, a suitable boundary condition is necessary. As we know, on the boundary $\partial\Omega$ of a large enough cut-off domain $\Omega$, both the value of \textit{pdf} and the flux across boundary are nearly zero. Therefore the numerical solution is indifferent to the choice of boundary condition, as long as it is homogeneous, i.e. one that admits the trivial solution, zero. We are free to choose one that simplifies the implementation and analysis later on.

Let $\normal$ denote the unit outward normal to $\Gamma\equiv \partial\Omega$. 



In the rest of this paper, we will use the following zero flux Robin condition on the boundary:
\begin{equation}
    (\mathcal{U}f-\mathcal{D}\cdot\nabla_\bp f)\cdot \normal=0,\text{ on }\Gamma,
    \label{eqn:bd_condition}
\end{equation}

We start motivating the LDG method by rewriting \eqref{eqn:adv_diff_form} in the following mixed form:

\begin{equation}
    \partial_t f=\nabla_\bp\cdot(Z-\mathcal{U}f).
\end{equation}
\begin{equation}
    Z=\mathcal{D}\cdot\widetilde{Z}
\end{equation}
with 
\begin{equation}\label{eq:auxiZ}
    \widetilde{Z}=\nabla_\bp f.
\end{equation}

Let the test space be $\mathcal{W}$, assuming smooth enough solutions, we get the following weak formulation of the equations above.
\begin{equation}
\left(\partial_{t}f, \varphi\right)_{\Omega} = \left(\mathcal{U}f - Z, \nabla_{p}\varphi\right)_{\Omega} 
    - \langle \mathcal{U}f - Z, \varphi^{-} \mathbf{n}^{-} \rangle_{\partial\Omega / \Gamma} , \quad \forall\ \varphi \in \mathcal{W}.
    \label{eqn:rhs_cont}
\end{equation}

\begin{equation}
    \left(Z,\widetilde{V}\right)_{\Omega}=\left(\mathcal{D}\widetilde{Z},\widetilde{V}\right)_{\Omega},\quad \forall\  \widetilde{V}\in(\mathcal{W})^3.
\end{equation}

\begin{equation}\label{eq:tildeZ}
    \left(\widetilde{Z},V\right)_{\Omega}=-\left(f,\nabla_\bp\cdot V\right)_{\Omega}+\langle f,V^{-}\cdot\mathbf{n}^{-}\rangle_{\partial \Omega},\quad \forall \ V\in (\mathcal{W})^3.
\end{equation}


Let $\partition=\{R\}$ be a partition of $\Omega$, with $R$ being Cartesian elements or simplices. Denote the mesh size of $\partition$ as $h=\max_{R\in\mathcal{T}_h}\mathrm{diam}(R)$.  Furthermore the inner products inside each element $R$ are defined as 
\begin{equation}
    (g,h)_{R}=\int_R gh\,d\mathbf{p}.
\end{equation}

Let $\mathcal{I}_{h}$ be the set of element faces, 
the integral on each face $e$ is denoted as
\begin{equation*}
    \langle g,h\rangle_{e}=\int_{e}gh\,d\mathbf{s}.
\end{equation*}

For piecewise functions defined with respect to $\partition$, we introduce the jumps and averages as follows. For any interface $e=\{R^+\cap R^-\}\in \mathcal{I}_{h}$, with $\normal^{\pm}$ as the outward unit normal to $\partial R^{\pm}$, $g^{\pm}=\left.g\right|_{R^{\pm}}$ , the jump across $e$ is defined as 
\begin{equation}
    [g]=g^{+}\normal^{+}+g^{-}\normal^{-}.
    \label{eqn:jump}
\end{equation}
and the average across $e$ is
\begin{equation}
    \{g\}=\frac{1}{2}(g^+ + g^-).
    \label{eqn:average}
\end{equation}

Next we define the discrete test space
\begin{equation}\label{eq:discrete_test_space}
    \mathcal{W}_h=\left\{g\in L^2(\Omega): \left.g\right|_{R}\in\Q^k(R), \forall R\in\partition\right\}.
\end{equation}
where $\Q^k(R)$ is the space of polynomials of degree at most $k$ in each variable.

By replacing the infinite dimensional test space $\mathcal{W}$ with the finite dimensional space $\mathcal{W}_{h}$, we obtain the semi-discrete weak form as follows: find $\partial_{t} f_{h} \in \mathcal{W}_{h}, Z_{h} \in \mathcal{W}_{h}^{3}, \widetilde{Z}_{h} \in \mathcal{W}_{h}^{3}$ such that
\begin{equation} \label{weakform-ldg}
    \begin{split}
        &\sum_{R}\left(\partial_{t}f_{h},\varphi_{h}\right)_{R} - \sum_{R}\left(\mathcal{U}f_{h}, \nabla_{p}\varphi_{h}\right)_{R} 
    + \sum_{R}\langle \widehat{\widehat{\mathcal{U}f_{h}}}, \varphi_{h}^{-} \mathbf{n}^{-} \rangle_{\partial R / \Gamma} \\
        =&-\sum_{R}\left(Z_{h},\nabla\varphi_{h}\right)_{R}+\sum_{R}\langle\widehat{Z}_{h},\varphi_{h}^{-}\mathbf{n}^{-}\rangle_{\partial R/\Gamma}, \forall \ \varphi_{h} \in  \mathcal{W}_{h},\\
        \sum_{R}\left(Z_{h},\widetilde{V}_{h}\right)_{R}
        =&\sum_{R}\left(\mathcal{D}\cdot\widetilde{Z}_{h},\widetilde{V}_{h}\right)_{R}, \forall\  \widetilde{V}_{h} \in \mathcal{W}_{h}^{3},\\
        \sum_{R}\left(\widetilde{Z}_{h},V_{h}\right)_{R}
        =&-\sum_{R}\left(f_{h},\nabla\cdot V_{h}\right)_{R}+\sum_{R}\langle\widecheck{f}_{h},V_{h}^{-}\cdot\mathbf{n}^{-}\rangle_{\partial R},\forall \  Z_{h} \in \mathcal{W}_{h}^{3}.
    \end{split}
\end{equation}

Since the discrete solutions $f_{h},Z_{h},\widetilde{Z}_{h}$ are discontinuous on the interfaces, the flux terms $\widehat{\widehat{\mathcal{U}f_{h}}}, \widehat{Z}_{h},\widecheck{f}_{h}$ remains to be determined. They will be addressed later in \eqref{Zflux}, \eqref{fflux}, and \eqref{upwindflux}.

\subsection{The bilinear form of the LDG method}

	
	

In the previous subsection, we introduced the LDG method, and left the key ingredient, flux terms, undetermined. Our conservation strategy is only possible with alternating flux. Before demonstrating the reason, we have to give a brief introduction on some prerequisite knowledge in this subsection.

In what follows, we will introduce a key concept: the \textit{discrete gradient operator} $\nabla_{h}$, and show that alternating flux enables a special bilinear form of LDG.

\bigskip
Choosing an appropriate reference direction $\mathbf{u}$, say $\mathbf{u}=(1,1,1)$, we use alternating flux as follows:
\begin{equation}\label{Zflux}
    \widehat{Z}=\begin{cases}
        Z^{-}, & \mathbf{n}^{-}\cdot\mathbf{u}>0\\
        Z^{+}, & \text{otherwise}
    \end{cases}
\end{equation}
and
\begin{equation}\label{fflux}
    \widecheck{f}=\begin{cases}
        f^{-}, & \mathbf{n}^{-}\cdot\mathbf{u}<0\ \text{or\ \ensuremath{\partial R}\ensuremath{\subset\partial\Omega}}\\
        f^{+}, & \text{otherwise}
    \end{cases}
\end{equation}

Now recall the definition of $\widetilde{Z}$ in Equation (\ref{eq:tildeZ}), it is just a projection of $\nabla_{p}f$ into the test space $\mathcal{W}_{h}$. In other words, $\widetilde{Z}$ is a discrete gradient of $f$. The map from $f$ to its discrete gradient $\widetilde{Z}$ is called the discrete gradient operator.

\begin{defi}[discrete gradient operator]\label{discretegrad}
We define $\nabla_{h}f_{h}$ as the function in $\mathcal{W}_{h}$ such that,
\begin{equation}\label{dgrad}
    \sum_{R}\left(\nabla_{h}f_{h},V_{h}\right)_{R}=-\sum_{R}\left(f_{h},\nabla\cdot V_{h}\right)_{R}+\sum_{R}\langle\widecheck{f}_{h},V_{h}^{-}\cdot\mathbf{n}^{-}\rangle_{\partial R},\forall V_{h}\in\mathcal{W}_{h}^{3}.
\end{equation}
And the operator $\nabla_{h}$ is the discrete gradient operator w.r.t. the flux $\widecheck{f}$ in Equation(\ref{fflux}).

\end{defi}

One must wonder what is the difference between a gradient operator and a discrete graident operator. The following proposition shows that for $f$ being discontinuous piecewise polynomial, the only difference comes from the jumps on the interfaces.

 \begin{prop}\label{u_equiv_d}
 If $f_{h} \in \mathcal{W}_{h}$, then $\nabla_{h} f_{h}$ satisfies that
\begin{equation}\label{dgrad2}
    -\sum_{R}\left(\nabla_{h}f_{h},V_{h}\right)_{R}=-\sum_{R}\left(\nabla f_{h},V_{h}\right)_{R}+\sum_{R}\langle f_{h}^{-}\mathbf{n}^{-},\widehat{V_{h}}\rangle_{\partial R/\Gamma},\forall V_{h}\in\mathcal{W}_{h}^{3}
\end{equation}
where $\widehat{V_{h}}$ is the alternating flux in Equation(\ref{Zflux}).

Consequently, if we define the lift operator $r$ as in \cite{arnold2002unified}:
\begin{equation*}
    \sum_{R}\left(r([[f_{h}]]),V_{h}\right)_{R}=-\sum_{l_{e}\in\partial R/\Gamma}\langle[[f_{h}]],\widehat{V_{h}}\rangle_{l_{e}},\forall V_{h}\in\mathcal{W}_{h}^{3},
\end{equation*}
then
\begin{equation}
    \nabla_{h}f_{h}=\nabla f_{h}+r([[f_{h}]]).
\end{equation}

\end{prop}
	
\begin{proof}
Integrate by parts on the right hand side of Equation(\ref{dgrad}),  we have
		\begin{equation*}
			\begin{split}
			    \sum_{R}\left(\nabla_{h}f_{h},V_{h}\right)_{R}=&-\sum_{R}\left(f_{h},\nabla\cdot V_{h}\right)_{R}+\sum_{R}\langle\widecheck{f_{h}},V_{h}^{-}\cdot\mathbf{n}^{-}\rangle_{\partial R}\\
       =&\sum_{R}\left(\nabla f_{h},V_{h}\right)_{R}+\sum_{R}\langle\widecheck{f}_{h}-f_{h}^{-},V_{h}^{-}\cdot\mathbf{n}^{-}\rangle_{\partial R/\Gamma}
			\end{split}
		\end{equation*}
		
Comparing with Equation(\ref{dgrad2}), it remains to show that
\begin{equation*}
    \sum_{R}\langle\widecheck{f}_{h}-f_{h}^{-},V_{h}^{-}\cdot\mathbf{n}^{-}\rangle_{\partial R/\Gamma}=-\sum_{R}\langle f_{h}^{-}\mathbf{n}^{-},\widehat{V_{h}}\rangle_{\partial R/\Gamma}.
\end{equation*}

Note that in the above summation, each interface is counted twice. Suppose that on the interface between $R_{l}$ and $R_{r}$ we have $\widecheck{f_{h}}=f^{l}_{h}$ and $\widehat{V_{h}}=V^{r}_{h}$, then

\begin{equation}
\begin{split}    
    &\sum_{R_{l}, R_{r}}\langle\widecheck{f_{h}},V_{h}^{-}\cdot\mathbf{n}^{-}\rangle_{\partial R_{l} \cap \partial R_{r}}+\sum_{R_{l}, R_{r}}\langle f_{h}^{-}\mathbf{n}^{-},\widehat{V_{h}}\rangle_{\partial R_{l} \cap \partial R_{r}}\\
    =&\langle f^{l}_{h},V_{h}^{l}\cdot\mathbf{n}^{l}\rangle_{\partial R_{l} \cap \partial R_{r}}+ \langle f^{l}_{h},V_{h}^{r}\cdot\mathbf{n}^{r}\rangle_{\partial R_{l} \cap \partial R_{r}}\\
    &+ \langle f_{h}^{l}\mathbf{n}^{l},V^{r}_{h}\rangle_{\partial R_{l} \cap \partial R_{r}} + \langle f_{h}^{r}\mathbf{n}^{r},V^{r}_{h}\rangle_{\partial R_{l} \cap \partial R_{r}}\\
    =&\langle f^{l}_{h},V_{h}^{l}\cdot\mathbf{n}^{l}\rangle_{\partial R_{l} \cap \partial R_{r}} + \langle f_{h}^{r}\mathbf{n}^{r},V^{r}_{h}\rangle_{\partial R_{l} \cap \partial R_{r}}
\end{split}
\end{equation}

It follows that
\begin{equation}
    \sum_{R}\langle\widecheck{f_{h}},V_{h}^{-}\cdot\mathbf{n}^{-}\rangle_{\partial R/\Gamma}+\sum_{R}\langle f_{h}^{-}\mathbf{n}^{-},\widehat{V_{h}}\rangle_{\partial R/\Gamma}=\sum_{R}\langle f_{h}^{-},V_{h}^{-}\cdot\mathbf{n}^{-}\rangle_{\partial R/\Gamma}
\end{equation}
\end{proof}

With discrete gradient operator, we will be able to rewrite the diffusion part in a compact bilinear form\cite{arnold2002unified}, without involving in the auxiliary field $\widetilde{Z}$ as in Equation(\ref{eq:auxiZ}).
	
\begin{theorem}
The whole system (\ref{weakform-ldg}) is equivalent to
\begin{equation}\label{semid_upwind}
\left(\partial_{t}f_{h},\varphi_{h}\right)_{\Omega}+\left(\nabla_{h}f_{h},\mathcal{D}\cdot\nabla_{h}\varphi_{h}\right)_{\Omega}=\left(\mathcal{U}f_{h},\nabla\varphi_{h}\right)_{\Omega}-\sum_{R}\langle\widehat{\widehat{\mathcal{U}f_{h}}},\varphi_{h}^{-}\mathbf{n}^{-}\rangle_{\partial R/\Gamma}
\end{equation}
\end{theorem}
	
\begin{proof}
To shorten the length of the formulas, we introduce the following notation:
\begin{equation*}
    \mathcal{A} = \left(\mathcal{U}f_{h},\nabla\varphi_{h}\right)_{\Omega}-\sum_{R}\langle\widehat{\widehat{\mathcal{U}f_{h}}},\varphi_{h}^{-}\mathbf{n}^{-}\rangle_{\partial R/\Gamma}
\end{equation*}

Recall the first row of Equation(\ref{weakform-ldg}), by Proposition \ref{u_equiv_d} we have:
\begin{equation*}
    \sum_{R}\left(\partial_{t}f_{h},\varphi_{h}\right)_{R}=-\sum_{R}\left(\nabla_{h}\varphi_{h},Z_{h}\right)_{R}+\mathcal{A}.
\end{equation*}

Since $\nabla_{h}\varphi \in \mathcal{V}$, it follows from the second row of Equation(\ref{weakform-ldg}) that 
\begin{equation*}
    \sum_{R}\left(\partial_{t}f_{h},\varphi_{h}\right)_{R}=-\sum_{R}\left(\mathcal{D}\cdot\widetilde{Z}_{h},\nabla_{h}\varphi_{h}\right)_{R}+\mathcal{A}.
\end{equation*}

By definition, the last row of Equation(\ref{weakform-ldg}) is equivalent to
		\begin{equation*}
			\widetilde{Z}_{h}=\nabla_{h} f_{h}
		\end{equation*}
		
Therefore, the diffusion part can be written as a bilinear form, and we have
\begin{equation*}
    \sum_{R}\left(\partial_{t}f_{h},\varphi_{h}\right)_{R}=-\sum_{R}\left(\mathcal{D}\cdot\nabla_{h}f_{h},\nabla_{h}\varphi_{h}\right)_{R}+\mathcal{A}.
\end{equation*}

\end{proof}

\section{Structure-Preserving Schemes}\label{sec:conservation_strategy}
In this section, we first introduce the concept of variational crime, as it helps understanding our strategy of structure-preserving. After that, we demonstrate how this strategy in general contradicts with stabilization, and propose an upwind structure-preserving LDG scheme in the end.

\subsection{Variational Crime}
A \textit{variational crime}\cite{strang1972variational} occurs when the discrete problem does not strictly follow the form of the continuous variational problem. For example, when solving the simple elliptic PDE 
\begin{equation*}
    \nabla\cdot(\mathcal{D}\cdot\nabla f) = \sigma 
\end{equation*}
with finite element method, the legit way is to solve the Galerkin variational problem: find $f_{h} \in \mathcal{W}_{h}$ such that
\begin{equation*}
    B(f_{h},\varphi_{h}) \coloneqq \int_{\Omega} \left(\mathcal{D}\cdot \nabla f_{h} \right)\cdot\nabla \varphi_{h}= S(\varphi) \coloneqq \int_{\Omega}\sigma \varphi_{h}, \forall \varphi_{h} \in \mathcal{W}_{h}.
\end{equation*}
However, in practice, it may be impossible to calculate the integral $B(f_{h}, \varphi_{h})$ exactly. And we have to compromise by approximating the integral with numerical quadrature, or approximating the diffusion coefficient $\mathcal{D}$ with piecewise polynomials to obtain an exact integral, etc.

Despite the ``crime", many such approximations still yield convergence, but they must be analyzed carefully.

In solving the FPL equation, variational crimes are usually unavoidable due to the complexity of the collision kernel $\Phi(\mathbf{p}, \mathbf{q})$, which makes exact integration infeasible and necessitates the use of quadrature. Since such variational crimes are inevitable, it is reasonable to take advantage of them to serve other purposes. 

It turns out that a structure-preserving scheme can be obtained by replacing the collision kernel $\Phi$ with a specifically designed approximation $\Phi_{h}$. 

\begin{remark}
    The same strategy not only applies to the FPL equation, but also to the quasilinear theory for particle-wave interaction\cite{huang2023conservative}.
\end{remark}

\subsection{Symmetric Structure-Preserving Scheme}

In \cite{shiroto2019structure}, Shiroto and Sentoku proposed a structure-preserving strategy for finite difference schemes which renders rigorous conservation laws. As we know, finite difference scheme is equivalent to discontinuous Galerkin scheme with piecewise constant basis functions. In this subsection, we generalize that strategy to LDG with higher order polynomials.

The conservation laws of FPL equations follow from a special structure in their weak forms. Recall the weak form in Equation \eqref{eq:weaklandau_structure}
\begin{equation*}
    \int_{\mathbb{R}^3}\varphi\partial_{t}f \,d \bp =-\int_{\mathbb{R}^3}\int_{\mathbb{R}^3} \Lambda(\bp,\bq) \mathbb{S}(\mathbf{v},\mathbf{w}):\left[\left(\nabla_{q}\varphi(\bq)-\nabla_{p}\varphi(\bp)\right)\otimes f(\bp)\nabla_{q}f(\bq)\right] \,d\bq \,d\bp .
\end{equation*}
Energy conservation relies on the fact that
\begin{equation*}
    \mathbb{S}\left(\nabla_{p}\mathcal{E}(\bp),\nabla_{q}\mathcal{E}(\bq)\right)\cdot\left(\nabla_{p}\mathcal{E}(\bp)-\nabla_{q}\mathcal{E}(\bq)\right)=0.
\end{equation*}

Considering that the single particle kinetic energy $\mathcal{E}(\mathbf{p})$ may not belong to the test function space $\mathcal{W}_{h}$, we pursue a numerical scheme that conserves discrete energy, i.e. a scheme such that
\begin{equation*}
    \partial_{t}\left(f_{h},\Pi_{h}\mathcal{E}\right)_{\Omega} = 0
\end{equation*}
where $\Pi_{h}$ is a projection to the test space $\mathcal{W}_{h}$, such that 
    \begin{align*}
        \lim_{h\rightarrow0}\lVert\mathcal{E}-\Pi_{h}\mathcal{E}\rVert_{L^{2}(\Omega)}&=0\\
        \lim_{h\rightarrow0,r(\Omega)\rightarrow \infty}\lVert \nabla\mathcal{E}-\nabla_{h}\Pi_{h}\mathcal{E}\rVert_{L^{2}(\Omega;e^{-p^2})}&=0
    \end{align*}

The following scheme is inspired by \cite{shiroto2019structure}.

\begin{theorem}[Symmetric Structure-Preserving Scheme]\label{thm:sp_not_upwind}
    If $f_{h}(\mathbf{p},t)$ is a solution to the following semi-discrete weak form,
    \begin{equation}\label{eq_cons_unstable}
		\int_{\Omega}\varphi_{h}\partial_{t}f_{h} \,d \bp =-\int_{\Omega}\int_{\Omega}\Phi_{h}(\mathbf{p},\mathbf{q}):\left[\left(\nabla_{h,q}\varphi_{h}(\mathbf{q})-\nabla_{h,p}\varphi_{h}(\mathbf{p})\right)\otimes f_{h}(\mathbf{p})\nabla_{h,q}f_{h}(\mathbf{q})\right] \,d\bq \,d\bp ,
	\end{equation}
 where $\nabla_{h,p}$ and $\nabla_{h,q}$ are the discrete gradient operator defined in Equation(\ref{dgrad}) and the discrete collision kernel is defined as 
 \begin{equation}\label{sp_kernel}
     \Phi_{h}(\mathbf{p},\mathbf{q})=\Lambda\left(\mathbf{p},\mathbf{q}\right)\mathbb{S}\left(\nabla_{h,p}\Pi_{h}\mathcal{E}(\mathbf{p}),\nabla_{h,q}\Pi_{h}\mathcal{E}(\mathbf{q})\right),
 \end{equation}
then the following discrete conservation laws hold:
\begin{equation}
    \begin{split}
        \partial_{t}\left(f_{h},1\right)_{\Omega}&=0\\
        \partial_{t}\left(f_{h},\Pi_{h}\mathbf{p}\right)_{\Omega}&=0\\
        \partial_{t}\left(f_{h},\Pi_{h}\mathcal{E}(\mathbf{p})\right)_{\Omega}&=0
    \end{split}
\end{equation}
\end{theorem}

\begin{proof}
    Mass and momentum conservation laws are trivial. For energy conservation, note that when substituting $\varphi_{h}$ with $\Pi_{h}\mathcal{E}(\mathbf{p})$ in Equation(\ref{eq_cons_unstable}), the integrand on the right hand side is proportional to the following term:
    \begin{equation}
        \mathbb{S}\left(\nabla_{h,p}\Pi_{h}\mathcal{E}(\mathbf{p}),\nabla_{h,q}\Pi_{h}\mathcal{E}(\mathbf{q})\right)\cdot\left(\nabla_{h,p}\Pi_{h}\mathcal{E}(\mathbf{p})-\nabla_{h,q}\Pi_{h}\mathcal{E}(\mathbf{q})\right),
    \end{equation}
    which is always equal to zero since for any $\mathbf{v},\mathbf{w} \in \mathbb{R}^3$, 
    \begin{equation*}
        \mathbb{S}\left(\mathbf{v},\mathbf{w}\right)\cdot\left(\mathbf{v}-\mathbf{w}\right)=0.
    \end{equation*}
\end{proof}

Note that the strategy does not rely on the choice of basis functions. For certain choices, it reproduces existing schemes:

\begin{itemize}
    \item For non-relativistic FPL equation, if all basis functions are continuous and the test space $\mathcal{W}_{h}$ contains the kinetic energy $\mathcal{E}(\mathbf{p})=\lvert\mathbf{p}\rvert^{2}/2$, then $\nabla_{h,p}\Pi_{h}\mathcal{E}(\mathbf{p})=\nabla_{p}\mathcal{E}(\mathbf{p})$, and \eqref{eq_cons_unstable} degenerates to a naive discretization as follows:
    
    \begin{equation}
        \int_{\Omega}\varphi_{h}\partial_{t}f_{h} \,d \bp =-\int_{\Omega}\int_{\Omega}\Phi(\mathbf{p},\mathbf{q}):\left[\left(\nabla_{q}\varphi_{h}(\mathbf{\bq})-\nabla_{p}\varphi_{h}(\mathbf{p})\right)\otimes f_{h}(\mathbf{p})\nabla_{q}f_{h}(\mathbf{q})\right] \,d\bq \,d\bp .
    \end{equation}
    
    \item For relativistic FPL equation, if all basis functions are continuous, then \eqref{eq_cons_unstable} recovers the finite element scheme proposed in \cite{hirvijoki2019conservative}.
    \item For relativistic FPL equation, if we use piecewise constant basis functions on a rectangular mesh in Cartesian coordinate system, then the discrete gradient operator can be explicitly written as 
    \begin{equation*}
        \nabla_{h,p}\varphi_{h}(\bp)=\left[\frac{\varphi_{h}^{i+1,j,k}-\varphi_{h}^{i,j,k}}{\Delta p},\frac{\varphi_{h}^{i,j+1,k}-\varphi_{h}^{i,j,k}}{\Delta p},\frac{\varphi_{h}^{i,j,k+1}-\varphi_{h}^{i,j,k}}{\Delta p}\right]
    \end{equation*}
    where $\Delta p$ is the mesh size and $i,j,k$ are the mesh indices. Our scheme \eqref{eq_cons_unstable} is then equivalent to the finite difference scheme proposed in \cite{shiroto2019structure}.
\end{itemize}

\subsection{Upwind Structure-Preserving Scheme}
As have been shown in Section {\ref{sec:background}}, the FPL equation can be regarded as a nonlinear nonlocal advection-diffusion equation. It is well-known that advection terms can cause numerical instability or sub-optimal convergence rate, unless the scheme is stabilized. 

Let us analyze the symmetric structure-preserving scheme from this new perspective. Note that \eqref{eq_cons_unstable} can be rephrased as follows:

\begin{equation}\label{eq:sp_au_form}
    \sum_{R}\left(\partial_{t}f_{h},\varphi_{h}\right)_{R}+\sum_{R}\left(\nabla_{h,p}f_{h},\mathcal{D}_{h}\cdot\nabla_{h,p}\varphi_{h}\right)_{R}=\sum_{R}\left(\mathcal{U}_{h}f_{h},\nabla_{h,p}\varphi_{h}\right)_{R}
\end{equation}
where $\mathcal{D}_{h}(\mathbf{p})=\int_{\Omega}\Phi_{h}(\mathbf{p},\mathbf{q})f_{h}(\mathbf{q}) \,d\bq $ and $\mathcal{U}_{h}(\mathbf{p})=\int_{\Omega}\Phi_{h}(\mathbf{p},\mathbf{q})\cdot\nabla_{h,q}f_{h}(\mathbf{q}) \,d\bq $.

Recall that the discrete gradient operator can be expressed as follows,
\begin{equation}
    \nabla_{h}\varphi_{h}=\nabla\varphi_{h}+r([[\varphi_{h}]]),
\end{equation}
therefore the advection part in Equation (\ref{eq:sp_au_form}) is apparently not upwind. Nevertheless, comparing it with the upwind scheme in Equation(\ref{semid_upwind}), the only difference comes from $[[\varphi_{h}]]$, the jump of test function across element interfaces. That means if mass, momentum and energy are continuous functions on $\Omega$, then upwind will not interfere with conservation. Based on this idea, we propose the following upwind structure-preserving method.

\begin{theorem}[Upwind Structure-Preserving Scheme]\label{thm:usp}
    Suppose there exists a non-zero projection $\Pi_{h}$ into $\mathcal{W}_{h}\cap C^{1}(\Omega)$. 
    Let $f_{h}(\mathbf{p},t)$ be a solution to the following semi-discrete weak form,
    \begin{equation}\label{eq:semid_upwind_sp}
    \sum_{R}\left(\partial_{t}f_{h},\varphi_{h}\right)_{R}+\sum_{R}\left(\nabla_{h}f_{h},\mathcal{D}_{h}\cdot\nabla_{h}\varphi_{h}\right)_{R}=\sum_{R}\left(\mathcal{U}_{h}f_{h},\nabla\varphi_{h}\right)_{R}-\sum_{R}\langle\widehat{\widehat{\mathcal{U}_{h}f_{h}}},\varphi_{h}^{-}\mathbf{n}^{-}\rangle_{\partial R/\Gamma}
    \end{equation}
    where $\mathcal{D}_{h}(\mathbf{p})=\int_{\Omega}\Phi_{h}(\mathbf{p},\mathbf{q})f_{h}(\mathbf{q})d \bq$ and $\mathcal{U}_{h}(\mathbf{p})=\int_{\Omega}\Phi_{h}(\mathbf{p},\mathbf{q})\cdot\nabla_{h,q}f_{h}(\mathbf{q}) d \bq$ and the kernel $\Phi_{h}$ is defined as in Equation (\ref{sp_kernel}). Then the following discrete conservation laws hold:
    \begin{equation}
        \begin{split}
            \partial_{t}\left(f_{h},\Pi_{h}1\right)_{\Omega}&=0\\
            \partial_{t}\left(f_{h},\Pi_{h}\mathbf{p}\right)_{\Omega}&=0\\
            \partial_{t}\left(f_{h},\Pi_{h}\mathcal{E}(\mathbf{p})\right)_{\Omega}&=0.
        \end{split}
    \end{equation}
    Moreover, the $L^{2}$ stability is enhanced as follows:
     \begin{equation}\label{eq:energy_id}
        \begin{split}
            &\partial_{t}\left(f_{h},f_{h}\right)_{\Omega}+\left(\nabla_{h}f_{h},\mathcal{D}_{h}\cdot\nabla_{h}f_{h}\right)_{\Omega}+\left(\frac{1}{2}\nabla\cdot\mathcal{U}_{h},f_{h}^{2}\right)_{\Omega}-\langle\frac{1}{2}\mathcal{U}_{h}\cdot\mathbf{n}^{-},(f_{h}^{-})^{2}\rangle_{\Gamma}\\
            =&-\sum_{l_{e}\in\partial R/\Gamma}\langle\frac{|\mathcal{U}_{h}\cdot\mathbf{n}|}{2}[[f_{h}]],[[f_{h}]]\rangle_{l_{e}}\\
            \leq&0.
        \end{split}
    \end{equation}
\end{theorem}

\begin{proof}
    Since the discrete energy $\Pi_{h}\mathcal{E}(\mathbf{p})$ belongs to $\mathcal{W}_{h}\cap C^{1}(\Omega)$, its discrete gradient must be continuous, i.e. $\nabla_{h,p}\Pi_{h}\mathcal{E}(\mathbf{p})\in C^{0}(\Omega)$. Consequently the tensor field $\mathbb{S}\left(\nabla_{h,p}\Pi_{h}\mathcal{E}(\mathbf{p}),\nabla_{h,q}\Pi_{h}\mathcal{E}(\mathbf{q})\right)$, as a composition of continuous functions, is continuous in both $\mathbf{p}$ and $\mathbf{q}$. Recall the definition of advection field $\mathcal{U}_{h}(\mathbf{p})$:
    \begin{equation}
        \mathcal{U}_{h}(\mathbf{p})=\int_{\Omega}\Phi_{h}(\mathbf{p},\mathbf{q})\cdot\nabla_{h,q}f_{h}(\mathbf{q}) d\bq=\int_{\Omega}\Lambda(\mathbf{p},\mathbf{q})\mathbb{S}\left(\nabla_{h,p}\Pi_{h}\mathcal{E}(\mathbf{p}),\nabla_{h,q}\Pi_{h}\mathcal{E}(\mathbf{q})\right)\cdot\nabla_{h,q}f_{h}(\mathbf{q}) d\bq
    \end{equation}
    It follows that $\mathcal{U}_{h}(\mathbf{p})\in C^{0}(\Omega)$. Hence the upwind flux
    \begin{equation}\label{upwindflux}
        \widehat{\widehat{\mathcal{U}_{h}f_{h}}}\cdot\mathbf{n}^{-}=\begin{cases}
\mathcal{U}_{h}f_{h}^{-}\cdot\mathbf{n}^{-}, & \mathcal{U}_{h}\cdot\mathbf{n}^{-}>0\\
\mathcal{U}_{h}f_{h}^{+}\cdot\mathbf{n}^{-}, & \mathcal{U}_{h}\cdot\mathbf{n}^{-}<0
\end{cases}
    \end{equation}
    is well-defined. Substitute $\varphi_{h}$ with $f_{h}$ in Equation (\ref{eq:semid_upwind_sp}), we obtain Equation (\ref{eq:energy_id}).

    Now let us consider the conservation laws. Recall that for any test function $\varphi_{h} \in \mathcal{W}_{h}$, its discrete gradient consists of the interior gradient part and the interface jump part:
    \begin{equation}
        \nabla_{h}\varphi_{h}=\nabla\varphi_{h}+r([[\varphi_{h}]])
    \end{equation}
    Hence the semi-discrete form (\ref{eq:semid_upwind_sp}) is equivalent to
     \begin{equation}\label{eq:semid_upwind_sp_jump}
        \begin{split}
            \partial_{t}\left(f_{h},\varphi_{h}\right)_{\Omega}=&-\int_{\Omega}\int_{\Omega}\Phi_{h}(\mathbf{p},\mathbf{q}):\left[\left(\nabla_{h,q}\varphi_{h}(\mathbf{q})-\nabla_{h,p}\varphi_{h}(\mathbf{p})\right)\otimes f_{h}(\mathbf{p})\nabla_{h,q}f_{h}(\mathbf{q})\right] d\bq d\bp \\
            &+\left\{ \left(\mathcal{U}_{h}f_{h},r([[\varphi_{h}]])\right)_{\Omega}-\sum_{l_{e}\in\partial R/\Gamma}\langle\mathcal{U}_{h}\{f_{h}\},[[\varphi_{h}]]\rangle_{l_{e}}-\sum_{l_{e}\in\partial R/\Gamma}\langle\frac{|\mathcal{U}_{h}\cdot\mathbf{n}|}{2}[[f_{h}]],[[\varphi_{h}]]\rangle_{l_{e}}\right\}
        \end{split}
    \end{equation}
    As have been proved in Theorem \ref{thm:sp_not_upwind}, the right hand side of the first row is equal to zero if $\varphi_{h} \in \{1,\Pi_{h}\mathbf{p},\Pi_{h}\mathcal{E}(\mathbf{p})\}$. Moreover, since $\{1,\Pi_{h}\mathbf{p},\Pi_{h}\mathcal{E}(\mathbf{p})\} \subset \mathcal{W}_{h}\cap C^{1}(\Omega)$, the jump terms in the second row will also vanish. In conclusion, if $\varphi_{h} \in \{1,\Pi_{h}\mathbf{p},\Pi_{h}\mathcal{E}(\mathbf{p})\}$,
     \begin{equation}
        \partial_{t}\left(f_{h},\varphi_{h}\right)_{\Omega}=0.
    \end{equation}
    
\end{proof}

\begin{remark}
The jump terms vanish as long as $\{1,\Pi_{h}\mathbf{p},\Pi_{h}\mathcal{E}(\mathbf{p})\} \subset C^{0}(\Omega)$, however we require $\Pi_{h}\mathcal{E}(\mathbf{p}) \in C^{1}(\Omega)$ to ensure the continuity of the advection field, which means the degree of polynomials $k$ should be at least $2$ in the test space \eqref{eq:discrete_test_space}. Of course, there may also be upwind numerical fluxes for discontinuous advection vector field. That is beyond the scope of this paper, and requires further investigation.
\end{remark}

\begin{remark}
It can be easily verified that our scheme is still locally conservative, since it can be regarded as traditional LDG applied to advection-diffusion equation with specially designed nonlinear nonlocal coefficients. 
\end{remark}

\section{Implementation}\label{sec:implement}
\subsection{Evaluation and Complexity}

As we know, the FPL collision operator is a bilinear operator, thus its weak form, as shown in Equation (\ref{eq:weaklandau_structure}), is a trilinear form:
\begin{equation*}
    \left(\partial_{t}f,\varphi\right)=\mathcal{B}(f,f,\varphi),
\end{equation*}
where
\begin{equation*}
    \mathcal{B}(f,g,\varphi)\coloneqq-\int_{\Omega}\int_{\Omega}\Phi(\bp,\bq):\left[\left(\nabla_{q}\varphi(\bq)-\nabla_{p}\varphi(\bp)\right)\otimes f(\bp)\nabla_{q}g(\bq)\right] d\bq d\bp.
\end{equation*}

According to Equation (\ref{eq:semid_upwind_sp}) and (\ref{eq:semid_upwind_sp_jump}), its upwind structure-preserving discretization $\mathcal{B}_{h}$ can be defined as:
\begin{equation}\label{eq:Bh}
    \begin{split}&\mathcal{B}_{h}\left(f_{h},g_{h},\varphi_{h}\right)\\
    \coloneqq& -\sum_{R}\left(\nabla_{h}g_{h},\mathcal{D}_{h}[f_{h}]\cdot\nabla_{h}\varphi_{h}\right)_{R}+\sum_{R}\left(\mathcal{U}_{h}[\nabla_{h}g_{h}]f_{h},\nabla\varphi_{h}\right)_{R}-\sum_{R}\langle\widehat{\widehat{\mathcal{U}_{h}[\nabla_{h}g_{h}]f_{h}}},\varphi_{h}^{-}\mathbf{n}^{-}\rangle_{\partial R/\Gamma}\\
    = & -\iint_{pq}\Phi_{h}(\mathbf{p},\mathbf{q}):\left[\left(\nabla_{h,q}\varphi_{h}(\mathbf{q})-\nabla_{h,p}\varphi_{h}(\mathbf{p})\right)\otimes f_{h}(\mathbf{p})\nabla_{h,q}g_{h}(\mathbf{q})\right]\\
 & +\{ \left(\mathcal{U}_{h}[\nabla_{h}g_{h}]f_{h},r([[\varphi_{h}]])\right)_{\Omega}-\sum_{l_{e}\in\partial R/\Gamma}\langle\mathcal{U}_{h}[\nabla_{h}g_{h}]\{f_{h}\},[[\varphi_{h}]]\rangle_{l_{e}}\\
 & -\sum_{l_{e}\in\partial R/\Gamma}\langle\frac{|\mathcal{U}_{h}[\nabla_{h}g_{h}]\cdot\mathbf{n}|}{2}[[f_{h}]],[[\varphi_{h}]]\rangle_{l_{e}}
\end{split}
\end{equation}

Suppose that the test space $\mathcal{W}_{h}$ is spanned by basis functions $\varphi_{i}$, then the solution $f_{h}$ can be expressed as follows:
\begin{equation}
    f_{h}(\mathbf{p},t)=\sum_{m=1}^{N}a_{m}(t)\varphi_{m}(\mathbf{p}).
\end{equation}

Note that as a consequence of upwinding, the discrete form $\mathcal{B}_{h}$ now contains a maximum function and an absolute value in the last row of \eqref{eq:Bh}, hence it is no longer linear w.r.t the function $g_{h}$. Therefore, during implementation, one can no longer precompute $B_{ijk} = \mathcal{B}_{h}\left(\varphi_{i},\varphi_{j},\varphi_{k}\right)$ and solve the ode system $\sum_{m}\left(\varphi_{m},\varphi_{k}\right)_{\Omega}d_{t}a_{m}(t)=\sum_{i,j}B_{ijk}a_{i}(t)a_{j}(t)$. Instead, we propose the following procedure:

\begin{enumerate}
    \item Given $f^{0}_{h}$, evaluate the fields $\mathcal{D}_{h}[f^{0}_{h}]$ and $\mathcal{U}_{h}[\nabla_{h} f^{0}_{h}]$ on Gauss-Lobatto quadrature points.
    \item Given $\Delta t$, solve Equation (\ref{eq:semid_upwind_sp}) element by element and obtain $f^{1}_{h}$.
    \item Let $f^{0}_{h}=f^{1}_{h}$ and repeat the above steps.
\end{enumerate}

Suppose that in each dimension we have $O(n)$ meshes, then the total number of elements and the degree of freedom $N$ are both in the order of $O(n^{3})$. In step 1, note that both $\mathcal{D}_{h}$ and $\mathcal{U}_{h}$ are linear maps from the the $N$-dimensional test space $\mathcal{W}_{h}$ to $\mathbb{R}^{N_{GL}}$, where the number of quadrature points $N_{GL}=O(n^{3})$, hence the time complexity for step 1 is $O(n^{6})$. Step 2 is just the normal time advancing step for DG methods, whose time complexity is known to be $O(n^{3})$.

\begin{remark}
    The space complexity for storing an $O(n^{3})\times O(n^{3})$ linear map is $O(n^{6})$. But in Section \ref{sec:numerics} we will show that for non-relativistic FPL equation, it can be reduced to $O(n^{3})$.
\end{remark}






\subsection{Temporal Discretizations}\label{ssec:temportal_discretization}
For the sake of numerical stability, we tend to use a high order Runge-Kutta scheme. As we know, RK schemes are a combination of Euler schemes, hence to begin with, we discuss the Euler schemes that preserve conservation laws. 

Analogous to \cite{huang2023conservative}, we have the following choices:

\begin{equation}
    \frac{1}{\Delta t}\left(f_{h}^{n+1}-f_{h}^{n},\varphi_{h}\right)=\begin{cases}
\mathcal{B}_{h}\left(f_{h}^{n+1},f_{h}^{n+1},\varphi_{h}\right), & \text{fully implicit},\\
\mathcal{B}_{h}\left(f_{h}^{n},f_{h}^{n+1},\varphi_{h}\right), & \text{semi-implicit},\\
\mathcal{B}_{h}\left(f_{h}^{n+1},f_{h}^{n},\varphi_{h}\right), & \text{semi-implicit},\\
\mathcal{B}_{h}\left(f_{h}^{n},f_{h}^{n},\varphi_{h}\right), & \text{fully explicit}.
\end{cases}
\end{equation}
The first two rows are not linear w.r.t $f^{n+1}$, therefore they are too expensive to be solved numerically. The third row seems favorable, however, recall the definition of semi-discrete weak form in Equation (\ref{eq:Bh}), we notice that it contains $\left(\nabla_{h}f_{h}^{n},\mathcal{D}_{h}[f_{h}^{n+1}]\cdot\nabla_{h}\varphi_{h}\right)$. Therefore, when analyzed as a scheme for advection-diffusion equation, it is explicit for the diffusion part but as expensive as an implicit scheme. In conclusion, our only choice is the fully explict scheme. Based on that, we choose the following RK3 scheme\cite{wanner1996solving}:
\begin{equation}
    \begin{split}
        \frac{1}{\Delta t}\left(w_{h}^{n}-f_{h}^{n},\varphi_{h}\right)&=\mathcal{B}_{h}\left(f_{h}^{n},f_{h}^{n},\varphi_{h}\right)\\
        \frac{1}{\Delta t}\left(y_{h}^{n}-\frac{1}{2}\left(f_{h}^{n}+w_{h}^{n}\right),\varphi_{h}\right)&=\mathcal{B}_{h}\left(w_{h}^{n},w_{h}^{n},\varphi_{h}\right)\\
        \frac{1}{\Delta t}\left(f_{h}^{n+1}-\frac{1}{3}\left(f_{h}^{n}+w_{h}^{n}+y_{h}^{n}\right),\varphi_{h}\right)&=\mathcal{B}_{h}\left(y_{h}^{n},y_{h}^{n},\varphi_{h}\right)
    \end{split}
\end{equation}


\section{Numerical Experiments}\label{sec:numerics}

The proposed upwind structure-preserving LDG scheme is a general framework for both relativistic and non-relativistic FPL equation, with any type of interaction potential. For simplicity, we choose the non-relativistic version as a test example. As will be shown in the following subsection, the space complexity can be reduced to $O(n^{3})$ instead of $O(n^{6})$, thanks to the convolution structure of the non-relativistic operator.

\subsection{The convolution structure of non-relativistic FPL operator}
Recall the non-relativistic FPL operator,
\begin{equation*}
		\mathcal{C}_{L}(f,f)=\nabla_{p}\cdot\left(\int_{\R^3}\Phi(\bp,\bq)\cdot\left(f(\bq)\nabla_{p}f(\bp)-f(\bp)\nabla_{q}f(\bq)\right)\,d\bq \right),
\end{equation*}
where
\begin{equation*}
    \begin{split}
        \Phi(\mathbf{p},\mathbf{q}) &= \Lambda(\mathbf{p},\mathbf{q})\mathbb{S}(\mathbf{p},\mathbf{q}),\\
        \Lambda(\mathbf{p},\mathbf{q}) &= \lvert\mathbf{p}-\mathbf{q}\rvert^{\gamma},\\
        \mathbb{S}(\mathbf{p},\mathbf{q})&=
        (\mathbf{p}-\mathbf{q})^2 - (\mathbf{p}-\mathbf{q})\otimes (\mathbf{p}-\mathbf{q}).
    \end{split}
\end{equation*}

Since the non-relativistic kinetic energy $\mathcal{E}=\mathbf{p}^{2}/2 \in \mathcal{W}_{h}$, we have
\begin{equation*}
    \Phi_{h}(\mathbf{p},\mathbf{q}) = \Phi(\mathbf{p},\mathbf{q}) = K(\mathbf{p}-\mathbf{q}).
\end{equation*}
Consequently
\begin{equation*}
    \mathcal{D}_{h}(\mathbf{p})=\int_{\Omega}K(\mathbf{p}-\mathbf{q})f_{h}(\mathbf{q}) \,d\bq ,
\end{equation*}
and 
\begin{equation*}
    \mathcal{U}_{h}(\mathbf{p})=\int_{\Omega}K(\mathbf{p}-\mathbf{q})\cdot\nabla_{h,q}f_{h}(\mathbf{q}) \,d\bq .
\end{equation*}

Since the solution $f_{h}$ and its discrete gradient $\nabla_{h}f_{h}$ are both linear combinations of basis functions $\varphi_{i}$'s, it is sufficient to precompute only the following matrix:
\begin{equation*}
    K_{i,j} = \int_{\Omega}K(\mathbf{p}_{i}-\mathbf{q}) \varphi_{j}(\mathbf{q})\,d\bq .
\end{equation*}

Although $K_{i,j}$ is a matrix of $O(n^{3})\times O(n^{3})$, we do not have to store all of the elements thanks to the convolution structure. Indeed, suppose that $\varphi_{j'}(\mathbf{q} - \mathbf{p}_{i'}) = \varphi_{j}(\mathbf{q} - \mathbf{p}_{i})$, it can be easily verified that
\begin{equation*}
    K_{i',j'} = K_{i,j}.
\end{equation*}
Consequently, when using uniform mesh, the actual space complexity for storing $K_{i,j}$ is in the order of $O(n^{3})$.

\subsection{Numerical Results}
We set the cutoff domain as $\Omega = (-4,4)^{3}$. The numerical experiment is performed on a $8\times8\times8$ uniform mesh. Let $\mathbb{Q}_{2} = \text{span}\{x^{i}y^{j}z^{k}\vert 0\leq i,j,k\leq2\}$, we take $\mathcal{W}_h=\left\{g\in L^2(\Omega): \left.g\right|_{R}\in\Q^{2}(R), \forall R\in\partition\right\}$ as our test space.  

Consider the following initial condition:
\begin{equation*}
    \left.f(p_{x}, p_{y}, p_{z},t)\right\vert_{t=0} = e^{-((p_{x}-1)^2 + p_{y}^2 + p_{z}^2)} + e^{-((p_{x}+1)^2 + p_{y}^2 + p_{z}^2)},
\end{equation*}
we solve the equation with Maxwell type interaction, i.e.
\begin{equation*}
    \Phi(\mathbf{p},\mathbf{q}) = (\mathbf{p}-\mathbf{q})^2 - (\mathbf{p}-\mathbf{q})\otimes (\mathbf{p}-\mathbf{q}).
\end{equation*}

The time evolution of the \textit{pdf} cross-section $\left.f(p_{x},p_{y},p_{z},t)\right\vert_{p_{z}=0}$ is shown in Figure \ref{fig:fevol}.

\begin{figure}[htbp!]
    \centering
    \begin{subfigure}[b]{0.45\textwidth}
        \centering
        \includegraphics[width=\textwidth]{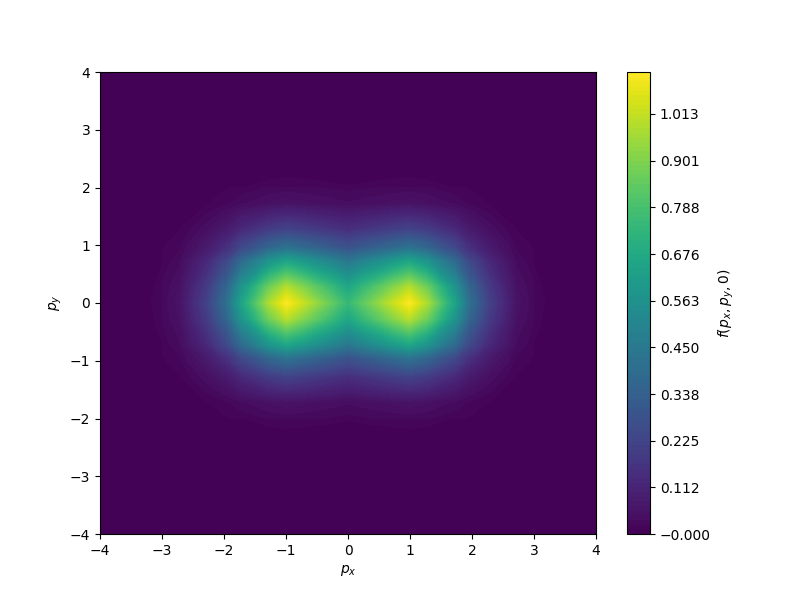}
        \caption{$f$ at $t = 0$}
    \end{subfigure}
    \begin{subfigure}[b]{0.45\textwidth}
        \centering
        \includegraphics[width=\textwidth]{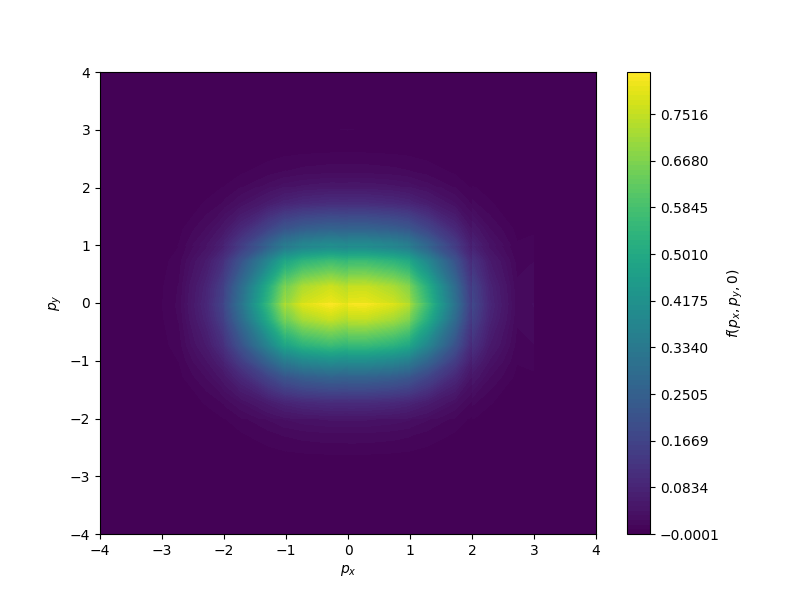}
        \caption{$f$ at $t = 0.013$}
    \end{subfigure}
    \begin{subfigure}[b]{0.45\textwidth}
        \centering
        \includegraphics[width=\textwidth]{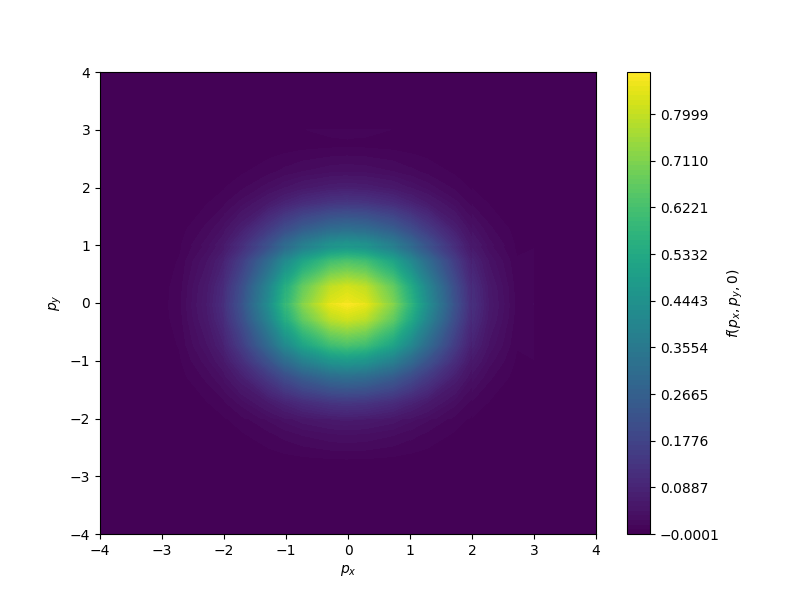}
        \caption{$f$ at $t = 0.027$}
    \end{subfigure}
    \begin{subfigure}[b]{0.45\textwidth}
        \centering
        \includegraphics[width=\textwidth]{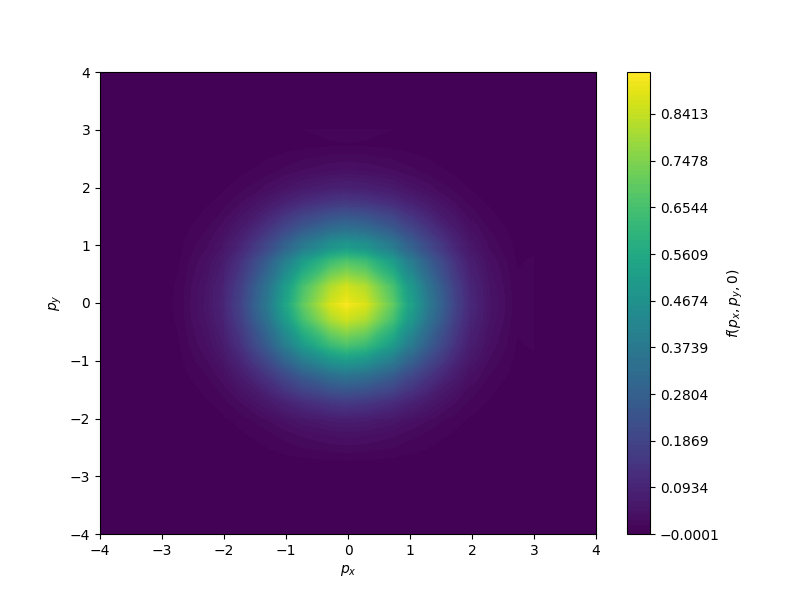}
        \caption{$f$ at $t = 0.041$}
    \end{subfigure}
    \caption{Cross section plot of $f(\mathbf{p},t)$ at $p_{z}=0$ in different time. The simulation was performed on a $8\times8\times8$ mesh with $\mathbb{Q}^{2}$ basis functions.}
    \label{fig:fevol}
\end{figure}

In Figure \ref{fig:moments} we also show the decay of relative entropy
\begin{equation*}
    \mathcal{H}(f\vert \mathcal{M}) = \int f^{+}\ln f^{+} - \int\mathcal{M}\ln \mathcal{M},
\end{equation*}
where the Maxwellian distribution $\mathcal{M}$ has the same mass, momentum and energy as the initial condition $\left.f(\mathbf{p},t)\right\vert_{t=0}$. Note that the scheme is not positive-preserving, so we use $f^{+}=\max(f,0)$ instead of $f$. The conservation errors are also shown in Figure \ref{fig:moments}.

\begin{figure}[htbp!]
    \centering
    \begin{subfigure}[b]{0.45\textwidth}
        \centering
        \includegraphics[width=\textwidth]{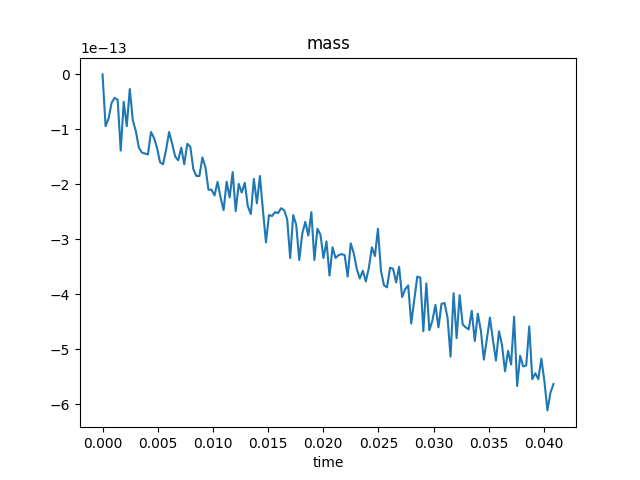}
    \end{subfigure}
    \begin{subfigure}[b]{0.45\textwidth}
        \centering
        \includegraphics[width=\textwidth]{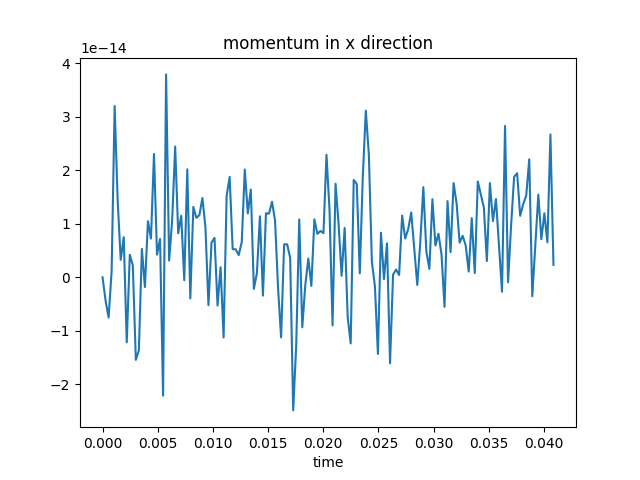}
    \end{subfigure}
    \begin{subfigure}[b]{0.45\textwidth}
        \centering
        \includegraphics[width=\textwidth]{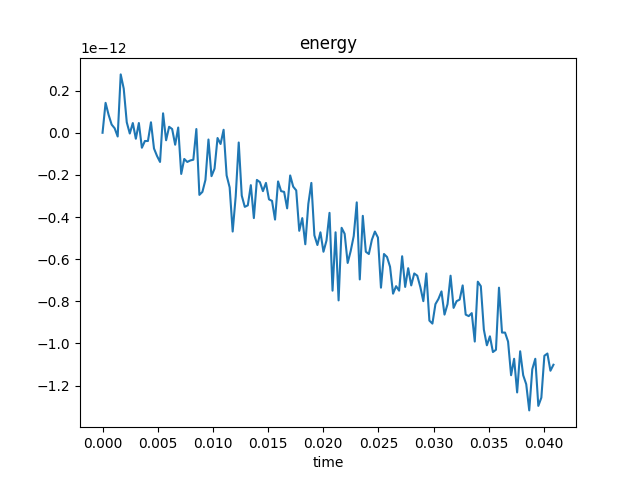}
    \end{subfigure}
    \begin{subfigure}[b]{0.45\textwidth}
        \centering
        \includegraphics[width=\textwidth]{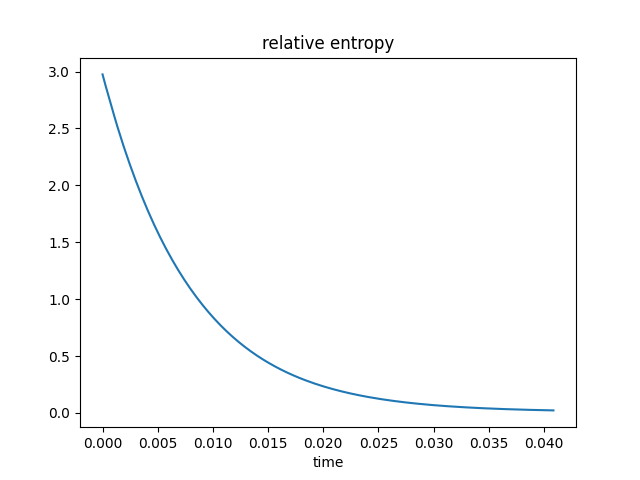}
    \end{subfigure}
    \caption{Conservation errors and the decay of entropy. The simulation was performed on a $8\times8\times8$ mesh with $\mathbb{Q}^{2}$ basis functions.}
    \label{fig:moments}
\end{figure}

\section{Concluding Remarks}\label{sec:conclude}
In this work, we developed a structure-preserving local discontinuous Galerkin (LDG) method for the Fokker–Planck–Landau (FPL) equation. By introducing a discrete gradient operator and carefully designing the collision kernel using projections of the particle energy, we constructed a scheme that simultaneously satisfies discrete conservation of mass, momentum, and energy. Unlike previous conservative discontinuous Galerkin methods, our approach ensures that the underlying geometric structure of the collision operator—specifically the null space of the tensor kernel—is preserved at the discrete level.

To resolve the contradiction between stability and conservation, we proposed an upwind structure-preserving scheme that exploits the unique treatment of jump terms in the LDG framework. The resulting method exhibits enhanced $L^2$ stability and is amenable to high-order discretizations and parallel implementation.

We validated our approach through numerical experiments on the non-relativistic FPL equation, confirming that the scheme accurately reproduces key physical behaviors such as entropy decay and convergence to equilibrium, while preserving all relevant conservation laws to machine precision.

In future work, we plan to apply this framework to the relativistic FPL equation, explore adaptive mesh refinement, and investigate coupling with Vlasov solvers in high-dimensional phase space.

\section{Acknowledgements}
The authors thank and gratefully acknowledge the support from the Oden Institute of Computational Engineering and Sciences, the Institute for Fusion Studies and the University of Texas Austin. The authors acknowledge the Texas Advanced Computing Center (TACC) at The University of Texas at Austin for providing computational resources that have contributed to the research results reported within this paper.

\bibliographystyle{plain}
\bibliography{main.bib}
\end{document}